\documentclass[]{article}
\usepackage{amssymb,amsmath,amsthm}

\newcommand{\Z}{{\mathbf Z}}
\newcommand{\R}{\mathbf{R}}

\renewcommand{\P}{\mathrm{P}}
\newcommand {\E}{\mathrm{E}}

\renewcommand{\d}{\text{\rm d}}
\renewcommand{\Re}{\text{\rm Re}\,}

\newcommand{\sL}{\mathcal{L}}
\newcommand{\sP}{\mathcal{P}}
\newcommand{\sG}{\mathcal{G}}
\newcommand{\sH}{\mathcal{H}}
\newcommand{\sA}{\mathcal{A}}
\newcommand{\lip}{\mathrm{Lip}}

\newtheorem{stat}{Statement}[section]
\newtheorem{proposition}[stat]{Proposition}
\newtheorem{corollary}[stat]{Corollary}
\newtheorem{theorem}[stat]{Theorem}
\newtheorem{lemma}[stat]{Lemma}
\theoremstyle{definition} 
\newtheorem{remark}[stat]{Remark}
\newtheorem{example}[stat]{Example}

\numberwithin{equation}{section}

\begin{document}

\title{\bf Intermittence and nonlinear  parabolic stochastic partial
	differential equations%
	\thanks{%
	Research supported in part by NSF grant DMS-0704024.}}
	
\author{Mohammud Foondun \and Davar Khoshnevisan\\
	\and University of Utah}

\date{April 16, 2008}
\maketitle
\begin{abstract}
	We consider nonlinear parabolic
	SPDEs of the form 
	$\partial_t u=\sL u + \sigma(u)\dot w$, where
	$\dot w$ denotes space-time white noise,
	$\sigma:\R\to\R$ is [globally] Lipschitz continuous,
	and $\sL$ is the $L^2$-generator of a L\'evy process.
	We present precise criteria for existence
	as well as uniqueness of solutions.
	More significantly, we prove that these solutions grow
	in time with at most a precise exponential rate.	
	We establish also that when $\sigma$ is globally Lipschitz
	and asymptotically sublinear, the solution
	to the nonlinear heat equation is ``weakly intermittent,''
	provided that the symmetrization
	of $\sL$ is recurrent and the initial data is sufficiently
	large.
	
	Among other things, 	our results lead to general
	formulas for the upper second-moment
	Liapounov exponent of the parabolic
	Anderson model for $\sL$ in dimension $(1+1)$. When
	$\sL=\kappa\partial_{xx}$ for $\kappa>0$, these formulas 
	agree with the earlier results of statistical physics
	\cite{Kardar,KrugSpohn,LL63}, and also probability theory
	\cite{BC,CM94} in the two exactly-solvable cases where $u_0=\delta_0$
	and $u_0\equiv 1$.\\
	
	\vskip .2cm \noindent{\it Keywords:}
		Stochastic partial differential equations, L\'evy processes,
		Liapounov exponents, weak intermittence, the
		Burkholder--Davis--Gundy inequality.\\
		
	\noindent{\it \noindent AMS 2000 subject classification:}
		Primary: 60H15; Secondary: 82B44.\\
		
	\noindent{\it Running Title:} Intermittency and 
		parabolic SPDEs.\newpage
\end{abstract}

\section{Introduction}

Let $\{\dot{w}(t\,,x)\}_{t\ge 0,x\in\R}$ denote space-time
white noise, and $\sigma:\R\to\R$ be a fixed Lipschitz function.
Presently we study parabolic stochastic partial differential
equations [SPDEs] of the following type:
\begin{equation}\label{heat}\left|\begin{split}
	&\partial_t u(t\,,x) = (\sL u)(t\,,x) + 
		\sigma(u(t\,,x))\dot{w}(t\,,x),\\
	&u(0\,,x)= u_0(x),
\end{split}\right.\end{equation}
where $t\ge 0$, $x\in\R$, $u_0$ is a measurable and nonnegative
initial function, and $\sL$
is the $L^2(\R)$-generator of a L\'evy process
$X:=\{X_t\}_{t\ge 0}$; and $\sL$ acts
only on the variable $x$.
We normalize $X$ so that
$\E\exp(i\xi X_t)=\exp(-t\Psi(\xi))$ for all $t\ge 0$ and
$\xi\in\R$; $\sL$ is
described via its Fourier multiplier as
$\hat\sL (\xi) =-\Psi(\xi)$ for all $\xi\in\R$.
See the books by Bertoin \cite{Bertoin} and Jacob \cite{Jacob}
for pedagogic accounts.

Our principal aim is to study the mild solutions of \eqref{heat},
when they exist. At this point in time, we understand
\eqref{heat} only when its linearization 
with vanishing inital data has
a strong solution.
Together with E. Nualart \cite{FKN},
we have investigated precisely those linearized equations.
That is,
\begin{equation}\label{heat:lin}\left|\begin{split}
	&\partial_t u(t\,,x) = (\sL u)(t\,,x) + 
		\dot{w}(t\,,x),\\
	&u(0\,,x)= 0.
\end{split}\right.\end{equation}
And we proved among other things  that \eqref{heat:lin}
has a strong solution if and only if Paul
L\'evy's symmetrization $\bar{X}$ of the process $X$
has local times, where
\begin{equation}\label{def:Xbar}
	\bar{X}_t := X_t-X_t'
	\qquad\text{for all $t\ge 0$},
\end{equation}
and $X':=\{X'_t\}_{t\ge 0}$ is an independent copy of $X$.
In fact, much of the local-time
theory of symmetric 1-dimensional L\'evy processes can
be embedded within the analysis of SPDEs defined by
\eqref{heat:lin}; see \cite{FKN} for details.
We also proved in \cite{FKN} that, as far as matters of existence
and regularity are concerned, one does not encounter new phenomena
if one adds to \eqref{heat:lin} Lipschitz-continuous
additive nonlinearities. This is why we consider 
only multiplicative nonlinearities in \eqref{heat}.

Let $\lip_\sigma$ denote the Lipschitz constant of
$\sigma$, and recall that $u_0$ is the initial
data in \eqref{heat}. Here and throughout we assume, without
further mention, that: 
\begin{enumerate}
	\item[(i)] $0<\lip_\sigma<\infty$, so that 
		$\sigma$ is [globally] Lipschitz and nontrivial; and
	\item[(ii)] $u_0$ is bounded, nonnegative, and measurable.
\end{enumerate}
Under these conditions, we prove that
the SPDE \eqref{heat} has a mild solution 
$u:=\{u(t\,,x)\}_{t\ge 0,x\in\R}$ that
is unique up to a modification. More significantly,
we show that
the growth of $t\mapsto u(t\,,x)$ is tied closely
with the existence of $u$.
With this aim in mind we
define the \emph{upper $p$th-moment Liapounov exponent}
$\bar\gamma(p)$ of $u$ as
\begin{equation}\label{def:gamma}
	\bar\gamma(p) := \limsup_{t\to\infty} \frac{1}{t}\ln
	\E\left(\left|u(t\,,x)\right|^p\right)
	\quad\text{for all $p\in(0\,,\infty)$}.
\end{equation}

We say that $u$ is \emph{weakly intermittent}\footnote{%
Our notion \eqref{def:WI} of weak intermittence
differs from that of \cite[Definition 1.2]{GartnerdenHollander}.} if
\begin{equation}\label{def:WI}
	\bar\gamma(2)>0 \quad\text{and}\quad
	\bar\gamma(p)<\infty\quad\text{for all $p>2$}.
\end{equation}
We are interested primarily in establishing
weak intermittence. However, let us mention also that 
weak intermittence can sometimes imply the much
better-known notion of
\emph{full intermittency} \cite[Definition III.1.1, p.\ 55]{CM94};
the latter is the property that
\begin{equation}\label{def:FI}
	p\mapsto\frac{\bar\gamma(p)}{p}\quad
	\text{is strictly increasing for all $p\ge 2$}.
\end{equation}
Here is a brief justification:
Evidently, $\bar\gamma$ is convex and zero at zero, and hence
$p\mapsto\bar\gamma(p)/p$ is nondecreasing. Convexity
implies readily that if in addition $\bar\gamma(1)=0$, then 
\eqref{def:WI} implies \eqref{def:FI}.\footnote{Inspect the proof
of Theorem III.1.2 in Carmona and Molchanov \cite[p.\ 55]{CM94}
for example.} On the other hand, 
a sufficient condition for $\bar\gamma(1)=0$
is that $u(t\,,x)\ge 0$ a.s.\ for all $t>0$
and $x\in\R$; for then, \eqref{mild} below shows
immediately that $\E(|u(t\,,x)|)=\E[u(t\,,x)]$
is bounded uniformly in $t$.
We have proved the following: 
``\emph{Whenever one has
a comparison principle---such as that of Mueller \cite{Mueller}
in the case that $\sL=\kappa\partial_{xx}$
and $\sigma(x)=\lambda x$---weak
intermittence necessarily implies full intermittency}.''

Here, we do not pursue comparison principles. Rather,
the principal goal of this note is to demonstrate that under 
various nearly-optimal conditions on $\sigma$ and $u_0$, the solution
$u$ to \eqref{heat} is weakly intermittent.

There is a big
literature on intermittency that investigates the special case
of \eqref{heat} with $\sL=\kappa\partial_{xx}$
and $\sigma(z)=\lambda z$ for constants $\kappa>0$
and $\lambda\in\R$; that is the \emph{parabolic Anderson
model}. See, for example,
\cite{BC,CM94,Kardar,KrugSpohn,LL63,Molch91}, together with their 
sizable combined references. The existing 
rigorous intermittency results all begin with
a probabilistic formulation of \eqref{heat} in terms of
the Feynman--Kac formula.
Presently, we introduce an analytic method
that shows clearly that weak
intermittence is connected intimately with the facts that:
(i) \eqref{heat} has a strong solution; and (ii) $\sigma$
has linear growth, in one form or another. Our method is
motivated very strongly by the theory of optimal regularity
for analytic semigroups \cite{Lunardi}.

We would like to mention also that there is an impressive body of
recent mathematical works on other Anderson
models and $L^p(\P)$ intermittency, as well as 
almost-sure intermittency
\cite[and their combined references]{CV98,CKM,CM,CMS,CMS1,DM,FV,%
GartnerdenHollander,GartnerKonig,GK,GKM,HKM,KLMS,Shiga}.

A brief outline follows: In \S\ref{sec:main} we
state the main results of the paper; these results are proved
subsequently in \S\ref{sec:proofs}, after we establish
some a priori bounds in \S\ref{sec:apriori}. Finally, 
we show in
Appendix \ref{sec:reg}  that if the initial data is
continuous, then the solution to \eqref{heat}
is continuous in probability, in fact 
continuous in $L^p(\P)$ for
all $p>0$. Consequently, if $u_0$ is
continuous, then $u$ has a separable modification.
As an immediate byproduct of our proof we find that when $\sL$ is
the fractional Laplacian of index $\alpha\in(1\,,2]$
and $u_0$ is continuous, $u$ has a jointly
H\"older-continuous modification (Example \ref{ex:Stable:Holder}).

\section{Main results}\label{sec:main}
We combine the existence result of
\cite{FKN} with a theorem of
Hawkes \cite{Hawkes} to deduce that \eqref{heat:lin}
has a strong solution if and only if
$\Upsilon(\beta)<\infty$ for some $\beta>0$, where
\begin{equation}\label{eq:Upsilon}
	\Upsilon(\beta)
	:= \frac{1}{2\pi}
	\int_{-\infty}^\infty \frac{\d\xi}{\beta+2\Re\Psi(\xi)}
	\qquad\text{for all $\beta>0$}.
\end{equation}
Furthermore, $\Upsilon(\beta)$ is finite for some $\beta>0$
if and only if it is finite for all $\beta>0$. And under
this integrability condition, \eqref{heat:lin} has a unique
solution as well. For related results, see 
Brze{\'z}niak and van Neerven \cite{BvN}.

Motivated by the preceding remarks, we consider only
the case that the linearized equation \eqref{heat:lin} 
has a strong solution. That is, we suppose
here and throughout that
$\Upsilon(\beta)<\infty$ for all $\beta>0$.
We might note that
$\Upsilon$ is decreasing, $\Upsilon(\beta)>0$
for all $\beta>0$, and
$\lim_{\beta\uparrow\infty}\Upsilon(\beta)=0$. 

Our next result establishes natural conditions for:
(i) the existence and uniqueness of a solution to
\eqref{heat}; and (ii) $u$ to grow at most exponentially
with a sharp exponent. It is possible to adapt the Hilbert-space
methods of  Peszat and Zabczyk \cite{PZ} to derive
existence and uniqueness. 
See also Da Prato \cite{Da} and Da Prato and Zabczyk \cite{DZ}.
Instead of following that route, we devise a 
method that shows very clearly that exponential growth
is a consequence of the existence of a solution. Moreover,
our method yields constants that will soon be shown to be
essentially unimproveable. 

Henceforth, by a ``solution'' to \eqref{heat} we mean a mild solution
$u$ that satisfies the following:
\begin{equation}
	\sup_{x\in\R}\sup_{t\in[0,T]}\E\left(\left| u(t\,,x)
	\right|^2\right)<\infty\quad
	\text{for all $T>0$}.
\end{equation}
It turns out that solutions to \eqref{heat} have better
\emph{a priori} integrability features. The following quantifies 
this remark.

\begin{theorem}\label{th:exist}
	Equation \eqref{heat} has a solution $u$
	that is unique up to a modification. Moreover,
	for all even integers $p\ge 2$,
	\begin{equation}\label{eq:UB}
		\bar\gamma(p) \le 
		\inf\left\{\beta>0:\, \Upsilon\left(\frac{2\beta}{p}\right)<
		\frac{1}{( z_p \lip_\sigma)^2}\right\}<\infty,
	\end{equation}
	where $ z_p$ denotes the largest 
	positive zero of the Hermite polynomial
	$\text{\sl He}_p$.
\end{theorem}

\begin{remark}\label{rem:He}
	We recall that $\text{\sl He}_p(x)=2^{-p/2}
	{\sl H}_p(x/2^{1/2})$
	for all $p>0$ and $x\in\R$,
	where $\exp(-2xt-t^2)=\sum_{k=0}^\infty {\sl H}_k(x)t^k/k!$
	for all $t>0$ and $x\in\R$.
	It is not hard to verify that
	\begin{equation}
		z_2 = 1
		\quad\text{and}\quad
		z_4=\sqrt{3+\sqrt{6}}\ \approx 2.334.
	\end{equation}
	This is valid simply because
	$\text{\sl He}_2(x) = x^2-1$ and
	$\text{\sl He}_4(x) = x^4-6x^2+3$. In addition,
	$ z_p\sim 2p^{1/2}$ as $p\to\infty$, and
	$\sup_{p\ge 1} ( z_p/p^{1/2})=2$;
	see Carlen and Kree \cite[Appendix]{CK}.\qed
\end{remark}

Before we explore the sharpness of \eqref{eq:UB}, 
let us examine two
cases that exhibit nonintermittence,
in fact \emph{subexponential growth}.
The first concerns \emph{subdiffusive growth}.

\begin{proposition}\label{pr:exist1}
	If $u_0$ and $\sigma$ are bounded and measurable, then
	for all integers $p\ge 2$,
	\begin{equation}
		\E \left( \left| u(t\,,x) \right|^p \right) =o
		\left( t^{p/2} \right)
		\qquad\text{as $t\to\infty$.}
	\end{equation}
\end{proposition}

\begin{remark}\label{rem:subdiff:sharp}
	The preceding is optimal; for instance,
	when $p=2$, the ``$o(t)$'' cannot 
	in general be improved to ``$o(t^\rho)$'' for any $\rho<1$.
	Indeed, consider the case that $\sL=-(-\Delta)^{\alpha/2}$
	is the fractional Laplacian. It is easy to see
	that $\Upsilon(\beta)<\infty$ for some $\beta>0$ iff
	$\alpha\in(1\,,2]$. If $0<\inf_{z\in\R}|\sigma(z)|
	\le\sup_{z\in\R}|\sigma(z)|<\infty$, then
	$\E ( | u(t\,,x) |^2 )$ is bounded above and below by
	constant multiples of $t^{(\alpha-1)/\alpha}$.
	We omit the details.\qed
\end{remark}

For our second proposition we first recall
the symmetrized L\'evy process $\bar{X}$ from \eqref{def:Xbar}.

\begin{proposition}\label{pr:trans}
	If $\bar{X}$ is transient, then for all 
	integers $p\ge 2$ there exists
	$\delta(p)>0$ such that $\bar\gamma(p)=0$ whenever
	$\lip_\sigma\!<\delta(p)$.
\end{proposition}

\begin{example}
	The conditions of Proposition \ref{pr:trans}
	are not vacuous.
	For instance, $\Psi(\xi)=|\xi|^\alpha+|\xi|^\rho$ is
	the exponent of a symmetric L\'evy process
	$\bar{X}$. Moreover, if $\alpha\in(0\,,1)$
	and $\rho\in(1\,,2]$, then $\bar{X}$ is
	transient and has local times.\qed
\end{example}

Our next result addresses the sharpness of 
\eqref{eq:UB}, and establishes an easy-to-check
sufficient criterion for $u$ to be
weakly intermittent. Throughout,
$\Upsilon^{-1}$ denotes the inverse to $\Upsilon$ in
the following sense:
\begin{equation}
	\Upsilon^{-1}(t) := \sup\left\{\beta>0:\
	\Upsilon(\beta)>t\right\},
\end{equation}
where $\sup\varnothing:=0$.

\begin{theorem}\label{th:behavior}
	If $\inf_{z\in\R} u_0(z)>0$ and
	$q:=\inf_{x\neq 0} |\sigma(x)/x|>0$, then
	\begin{equation}\label{eq:BE}
		\bar\gamma(2) \ge 
		\Upsilon^{-1}\left( \frac{1}{q^2}\right)>0.
	\end{equation}
\end{theorem}

Our next result is a ready corollary of Theorems \ref{th:exist}
and \ref{th:behavior}; see Carmona and Molchanov \cite[p.\ 59]{CM94},
Cranston and Molchanov \cite{CM1}, and
G\"artner and den Hollander \cite{GartnerdenHollander} for 
phenomenogically-similar results.
It might help to recall \eqref{def:Xbar}.

\begin{corollary}\label{cor:L:Anderson}
	If $\sigma(x):=\lambda x$ and $\inf_{x\in\R}u_0(x)>0$,
	then:
	\begin{enumerate}
	\item
		If $\bar{X}$ is recurrent, then $u$
		is weakly intermittent;
	\item
		If $\bar{X}$ is transient,
		then $u$ is weakly intermittent if and only if
		$\Upsilon(\beta)\ge \lambda^{-2}$ for some $\beta>0$; and
	\item  In all the cases that $u$ is weakly intermittent, 
		$\bar\gamma(2)= \Upsilon^{-1} (\lambda^{-2})$.
	\end{enumerate}
\end{corollary}

Even though Corollary \ref{cor:L:Anderson} is concerned with
a very special case of \eqref{heat}, that special case
has a rich history. Indeed, Corollary \ref{cor:L:Anderson}
contains a moment analysis of 
the socalled \emph{parabolic Anderson model}
for $\sL$. When $\sL = \kappa \partial_{xx}$,
that equation arises in
the analysis of branching processes in random environment 
\cite{CM94,Molch91}. If the spatial motion is a L\'evy process
with generator $\sL$, then we arrive at
\eqref{heat} with $\sigma(x)=\lambda x$.
For somewhat related---though not identical---reasons,
the parabolic Anderson model also
paves the way for a mathematical understanding of
the socalled ``KPZ equation'' in dimension
$(1+1)$. For further information see the original paper
by Kardar, Parisi, and Zhang \cite{KPZ}, Chapter 5 of Krug and Spohn
\cite{KrugSpohn}, and the Introduction by 
Carmona and Molchanov \cite{CM94}.

\begin{example}\label{ex:L:Anderson}
	If the conditions of Corollary \ref{cor:L:Anderson}
	hold, then the solution to \eqref{heat} with 
	$\sL=-\kappa(-\Delta)^{\alpha/2}$ is weakly intermittent with
	\begin{equation}\label{eq:E2:stable}
		\bar\gamma(2)=\left(\frac{\nu^\alpha\lambda^{2\alpha}}{
		\kappa}\right)^{1/(\alpha-1)}\!\! \text{where}\quad
		\nu := \frac{\sec(\pi/\alpha)}{2^{1/\alpha}\alpha}.
	\end{equation}
	Of course, we need $\alpha\in(1\,,2]$, and this implies
	that $\bar{X}$ is recurrent; see Remark
	\ref{rem:subdiff:sharp}. In order to derive
	\eqref{eq:E2:stable}, we first recall that
	$\int_0^\infty \d x/(1+x^\alpha)
	=(\pi/\alpha)\sec (\pi/\alpha)$.
	Thus, a direct computation yields
	$\Upsilon(\beta) = \nu\kappa^{-1/\alpha}\,\beta^{-1+(1/\alpha)}$
	for all $\beta>0$.
	Corollary \ref{cor:L:Anderson}, and a few more
	simple calculations, together imply \eqref{eq:E2:stable}.
	A similar argument shows that
	\begin{equation}\label{eq:p_gamma}
		\bar\gamma(p) \le \frac{p}{2}\left(\frac{\nu^\alpha}{\kappa}
		\left(z_p\lambda\right)^{2\alpha}\right)^{1/(\alpha-1)}
		\quad\text{for all even integers $p\ge 2$}.
	\end{equation}
	We can use this in conjunction with the Carlen--Kree inequality
	$[z_p\le 2\sqrt p$; see Remark \ref{rem:He}] to obtain 
	explicit numerical bounds.
	\qed
\end{example}

In the special case that $\sL=\kappa \partial_{xx}$,
Example \ref{ex:L:Anderson} tells that $\bar\gamma(2)
=\lambda^4/(8\kappa)$ for all $x\in\R$. This formula
is anticipated by the earlier investigations 
of Lieb and Liniger \cite{LL63}
and Kardar \cite[Eq.\ (2.9)]{Kardar} in statistical physics;
it can also be deduced upon combining the results
of Bertini and Cancrini \cite{BC}, in the exact case $u_0\equiv 1$,
with Mueller's comparison principle \cite{Mueller}.
Carmona and Molchanov \cite[p.\ 59]{CM94} study a closely-related
parabolic Anderson model in which $\dot{w}(t\,,x)$
is white noise over $(t\,,x)\in\R_+\times\Z^d$.

It is also easy to see that the bound furnished by
\eqref{eq:p_gamma} is nearly sharp in the case that $\alpha=2$
and $p>2$.
For example, \eqref{eq:p_gamma} and the Carlen--Kree
inequality $[z_p\le 2\sqrt p]$ together yield
$\bar\gamma(p)\le \vartheta(p):= (p^3\lambda^4)/\kappa$,
valid for all even integers $p\ge 2$.
When $p\ge 2$ is an arbitrary integer,
the exact answer is $\bar\gamma(p)= p(p^2-1)\lambda^4/(48\kappa)$
\cite{BC,Kardar,LL63}, and the $\limsup$ in the definition
of $\bar\gamma$ is a bona fide limit.
Our bound $\vartheta(p)$ agrees well with the
exact answer in this special case. Indeed,
\begin{equation}
	1\le\frac{\vartheta(p)}{\bar\gamma(p)}\le 48\left(
	1+\frac{1}{p^2-1}\right),
\end{equation}
uniformly for all even integers $p\ge 2$, as well as all
$\lambda\in\R$ and $\kappa\in(0\,,\infty)$.

We close with a result that states roughly that if
$\sigma$ is asymptotically linear and $\bar{X}$ is
recurrent, then a sufficiently large initial data will
ensure intermittence. More
precisely, we have the following.

\begin{theorem}\label{th:sublinear}
	Suppose $\bar{X}$ is recurrent, and
	$q:=\liminf_{|x|\to\infty}
	|\sigma(x)/x|>0$.
	Then, there exists $\eta_0>0$
	such that whenever $\eta:=\inf_{x\in\R}u_0(x)\ge\eta_0$,
	the solution $u$ is weakly intermittent.
\end{theorem}

We believe this result presents a notable improvement
on the content of Theorem \ref{th:behavior} in the case
that $\bar{X}$ is recurrent.

\section{A priori bounds}\label{sec:apriori}
Before we prove the mathematical
assertions of \S\ref{sec:main},
let us develop some of the required background.
Throughout we note the following elementary bound:
\begin{equation}\label{eq:sigma:linear}
	|\sigma(x)|\le |\sigma(0)|+\lip_\sigma|x|
	\qquad\text{for all $x\in\R$}.
\end{equation}

Define $\{\sP_t\}_{t\ge 0}$ as the semigroup associated with
$\sL$.
According to Lemma 8.1 of Foondun et al.\
\cite{FKN}, there exist transition densities
$\{p_t\}_{t> 0}$, whence we have
\begin{equation}
	(\sP_tg)(x) = \int_{-\infty}^\infty p_t(y-x)g(y)\,\d y
	\qquad(t>0).
\end{equation}
For us, the following relation is also significant:
$(\sP_t^*g)(x) = (\breve{p}_t*g)(x)$,
where $\sP_t^*$ denotes the adjoint of $\sP_t$ in $L^2(\R)$,
and $\breve{p}_t(x):=p_t(-x)$.

Consider
\begin{equation}
	(\sG u_0)(t\,,x) := (\sP_tu_0)(x)=(\breve{p}_t*u_0)(x)\qquad
	\text{for all $t>0$ and $x\in\R$.}
\end{equation}
We define also $(\sG u_0)(0\,,x):=u_0(x)$ for all $x\in\R$.
The function $v=\sG u_0$ solves the nonrandom integro-differential
equation
\begin{equation}\label{heat:nonrandom}\left|\begin{split}
	&\partial_t v = \sL v\hskip1in \text{on $(0\,,\infty)\times\R$},\\
	&v(0\,,x) = u_0(x)\hskip0.62in\text{for all $x\in\R$}.
\end{split}\right.\end{equation}
Thus, we can follow the terminology and methods of Walsh
\cite{Walsh} closely to deduce that
\eqref{heat} admits a mild solution $u$ if and only if $u$ 
is a predictable process that solves
\begin{equation}\label{mild}
	u(t\,,x) = (\sG u_0)(t\,,x)
	+\int_{-\infty}^\infty\int_0^t 
	\sigma(u(s\,,y)) p_{t-s}(y-x)\, w(\d s\,\d y).
\end{equation}

We begin by making two simple computations. The first
is a basic potential-theoretic bound.

\begin{lemma}\label{lem:monotone}
	For all $\beta>0$,
	\begin{equation}\label{eq:monotone}
		\sup_{t>0} e^{-\beta t}\int_0^t\|p_s\|_{L^2(\R)}^2\,\d s
		\le 
		\int_0^\infty e^{-\beta s} \|p_s\|_{L^2(\R)}^2\,\d s
		=\Upsilon(\beta).
	\end{equation}
\end{lemma}

\begin{proof}
	The inequality is obvious; we
	apply Plancherel's theorem to find that
	\begin{equation}\label{eq:L2p}
		\|p_s\|_{L^2(\R)}^2 = \frac{1}{2\pi}\int_{-\infty}^\infty
		e^{-2s\Re\Psi(\xi)}\,\d\xi\qquad\text{for all $s>0$}.
	\end{equation}
	Therefore, Tonelli's theorem implies the remaining equality.
\end{proof}

Next we present our second elementary estimate.

\begin{lemma}\label{lem:arith}
	For all $a,b\in\R$ and $\epsilon>0$,
	\begin{equation}
		(a+b)^2\le (1+\epsilon)a^2+ \left( 1+\epsilon^{-1}
		\right)b^2.
	\end{equation}
\end{lemma}

\begin{proof}
	Define $h(\epsilon)$ to be the upper bound of the lemma.
	Then, $h:(0\,,\infty)\to\R_+$ is minimized at
	$\epsilon = |b/a|$, and the minimum value of $h$
	is $a^2+2|ab|+b^2$, which is in turn $\ge(a+b)^2$.
\end{proof}

Now we proceed to establish the remaining
requires estimates. 

For every positive $t$ and all Borel sets
$A\subset\R$, we set $w_t(A):=\dot w([0\,,t]\times A)$,
and let $\mathcal{F}_t$
denote the $\sigma$-algebra generated by all
Wiener integrals of the form $\int g(x)\, w_s(\d x)$,
as the function $g$ ranges over $L^2(\R)$ and 
the real number $s$ ranges over $[0\,,t]$. Without loss of too much
generality we may assume that the resulting filtration
$\mathcal{F}:=\{\mathcal{F}_t\}_{t\ge 0}$ satisfies
the usual conditions, else we enlarge each $\mathcal{F}_t$
in the standard way. Here and throughout, a process
is said to be \emph{predictable} if  it is predictable
with respect to $\mathcal{F}$; see also Walsh \cite[p.\ 292]{Walsh}.

Given a predictable random field $f$, we define
\begin{equation}\label{def:A}
	(\sA  f)(t\,,x) := \int_{-\infty}^\infty\int_0^t
	\sigma(f(s\,,y)) p_{t-s}(y-x)\, w(\d s\,\d y),
\end{equation}
for all $t\ge 0$ and $x\in\R$,
provided that the stochastic integral exists in the sense
of Walsh \cite{Walsh}.
We also define a family of $p$-norms $\{\|f\|_{p,\beta}\}_{\beta>0}$,
one for each integer $p\ge 2$, via
\begin{equation}
	\|f\|_{p,\beta} := \left\{ \sup_{t\ge 0}\sup_{%
	x\in\R} e^{-\beta t}\E\left(\left|
	f(t\,,x)\right|^p\right)\right\}^{1/p}.
\end{equation}
Variants of these norms appear in several places in the SPDE
literature. See, in particular, Peszat and Zabczyk \cite{PZ}.
However, there is a subtle [but very important!] novelty here:
The supremum is taken over all time.

Recall the definition of $z_p$ from Theorem \ref{th:exist}.

\begin{lemma}\label{lem:A1}
	If $f$ is predictable and
	$\|f\|_{p,\beta}<\infty$ for a real $\beta>0$
	and an even integer $p\ge 2$, then
	\begin{equation}
		\|\sA f\|_{p,\beta}
		\le z_p\left( |\sigma(0)|
		+ \lip_\sigma \|f\|_{p,\beta}
		\right) \sqrt{\Upsilon\left(\frac{2\beta}{p}\right)}.
	\end{equation}
\end{lemma}

\begin{proof}
	In his seminal 1976 paper \cite{Davis}, Burgess Davis found the
	optimal constants in the Burkholder--Davis--Gundy [BDG] inequality.
	In particular, Davis proved that for all $t\ge 0$ and $p\ge 2$,
	\begin{equation}
		z_p = \sup\left\{
		\frac{\|N_t\|_{L^p(\P)}}{\|
		\langle N\,,N\rangle_t\|_{L^{p/2}(\P)}^2}:
		\, N\in\mathfrak{M}_p
		\right\},
	\end{equation}
	where $0/0:=0$ and $\mathfrak{M}_p$ denotes the collection of all
	continuous $L^p(\P)$-martingales.
	We apply Davis's form of the BDG inequality 
	[\emph{loc.\ cit.}], and find that
	\begin{equation}\begin{split}
		&\|( \sA f)(t\,,x)\|_{L^p(\P)}^p\\
		&\hskip.8in\le z_p^p\E\left(\left|
			\int_{-\infty}^\infty\d y\int_0^t\d s\,
			\left| \sigma(f(s\,,y))\right|^2
			\left|p_{t-s}(y-x)\right|^2
			\right|^{p/2} \right).
	\end{split}\end{equation}
	Since $p/ 2$ is a positive integer, the preceding expectation can
	be written as
	\begin{equation}
		\E\left(\prod_{j=1}^{p/2} \int_{-\infty}^\infty\d y_j
		\int_0^t\d s_j\,
		\left| \sigma(f(s_j\,,y_j))\right|^2
		\left|p_{t-s_j}(y_j-x)\right|^2 \right).
	\end{equation}
	The generalized H\"older inequality tells us that
	\begin{equation}\label{eq:Holder}
		\E\left(\prod_{j=1}^{p/2}\left| \sigma(f(s_j\,,y_j))\right|^2
		\right)\le \prod_{j=1}^{p/2}
		\left\| \sigma(f(s_j\,,y_j)) \right\|_{L^p(\P)}^2.
	\end{equation}
	Therefore, a little algebra shows us that
	\begin{equation}
		\|( \sA f)(t\,,x)\|_{L^p(\P)}^2
		\le z_p^2 \int_{-\infty}^\infty\d y
		\int_0^t\d s\,
		\left\| \sigma(f(s\,,y))\right\|_{L^p(\P)}^2
		\left|p_{t-s}(y-x)\right|^2.
	\end{equation}
	Owing to \eqref{eq:sigma:linear} and Minkowski's inequality,
	\begin{equation}\begin{split}
		&\|( \sA f)(t\,,x)\|_{L^p(\P)}^2\\
		&\hskip.4in\le z_p^2 \int_{-\infty}^\infty\d y
			\int_0^t\d s\,
			\left( c_0+c_1\| f(s\,,y)\|_{L^p(\P)}\right)^2
			\left|p_{t-s}(y-x)\right|^2,
	\end{split}\end{equation}
	where $c_0:=|\sigma(0)|$ and $c_1:=\lip_\sigma$,
	for brevity.
	Therefore, Lemmas \ref{lem:monotone}
	and \ref{lem:arith} together imply the following bound,
	valid for all $\beta>0$:
	\begin{align}
		\|( \sA f)(t\,,x)\|_{L^p(\P)}^2
			&\le \left(1+\epsilon^{-1}\right)z_p^2c_0^2
			e^{2\beta t/p}\Upsilon\left(\frac{2\beta}{p}\right)\\\nonumber
		&\quad+(1+\epsilon)z_p^2c_1^2
			\int_{-\infty}^\infty\d y
			\int_0^t\d s\,\|f(s\,,y)\|_{L^p(\P)}^2
			\left| p_{t-s}(y-x)\right|^2.
	\end{align}
	Because $\|f(s\,,y)\|_{L^p(\P)}^2 \le
	\exp(2\beta s/p)\|f\|_{p,\beta}^2$,
	it follows that
	\begin{equation}\begin{split}
		&\|( \sA f)(t\,,x)\|_{L^p(\P)}^2\\
		&\ \le \left(1+\epsilon^{-1}\right)z_p^2c_0^2
			e^{2\beta t/p}\Upsilon\left(\frac{2\beta}{p}\right)\\
		&\hskip.5in+(1+\epsilon)z_p^2c_1^2e^{2\beta t/p}\|f\|_{p,\beta}^2
			\int_{-\infty}^\infty\d y\int_0^t\d s\,
			e^{-2\beta s/p}\left| p_s(y-x)\right|^2\\
		&\ \le \left(1+\epsilon^{-1}\right)z_p^2c_0^2
			e^{2\beta t/p}\Upsilon\left(\frac{2\beta}{p}\right)
			+(1+\epsilon)z_p^2c_1^2e^{2\beta t/p}\|f\|_{p,\beta}^2
			\Upsilon\left(\frac{2\beta}{p}\right).
	\end{split}\end{equation}
	See Lemma \ref{lem:monotone} for the final inequality.
	We multiply both sides by
	$\exp(-2\beta t/p)$ and optimize over $t\ge 0$
	and $x\in\R$ to deduce the estimate
	\begin{equation}
		\|\sA f\|_{p,\beta}^2
		\le z_p^2\left\{\left(1+\epsilon^{-1}\right)|\sigma(0)|^2
		+(1+\epsilon) \lip_\sigma^2 \|f\|_{p,\beta}^2
		\right\}\Upsilon\left(\frac{2\beta}{p}\right).
	\end{equation}
	The preceding is valid for all $\epsilon>0$. Now we choose
	\begin{equation}
		\epsilon:=
		\begin{cases}
			|\sigma(0)|/(\lip_\sigma\|f\|_{p,\beta})&\text{if $
				|\sigma(0)|\cdot
				\|f\|_{p,\beta}>0$},\\
			0&\text{if $\sigma(0)=0$},\\
			\infty&\text{if $\|f\|_{p,\beta}=0$},
		\end{cases}
	\end{equation}
	to arrive at the statement of the lemma. Of course,
	``$\epsilon=\infty$'' means ``send $\epsilon\to\infty$''
	in the preceding.
\end{proof}

We plan to carry out a fixed-point argument in order
to prove Theorem \ref{th:exist}. The following result
shows that the stochastic-integral operator
$f\mapsto \sA f$ is a contraction on
suitably-chosen spaces.

\begin{lemma}\label{lem:A2}
	Choose and fix an even integer $p\ge 2$.
	For every $\beta>0$, and all predictable random fields
	$f$ and $g$ that satisfy $\|f\|_{p,\beta}+\|g\|_{p,\beta}<\infty$,
	\begin{equation}
		\| \sA f-\sA g \|_{p,\beta}
		\le z_p\lip_\sigma \sqrt{\Upsilon\left(\frac{2\beta}{p}\right)}
		\| f-g\|_{p,\beta}.
	\end{equation}
\end{lemma}

\begin{proof}
	The proof is a variant of the preceding argument.
	Namely,
	\begin{align} \nonumber
		&\E\left(\left| (\sA f)(t\,,x)-(\sA g)(t\,,x)\right|^p\right)\\
		&\le z_p^p \E\left(\left| \int_{-\infty}^\infty\d y
			\int_0^t\d s\,
			\left| \sigma(f(s\,,y))-\sigma(g(s\,,y))\right|^2
			\left| p_{t-s}(y-x)\right|^2 \right|^{p/2}\right)\\\nonumber
		&\le \left(z_p\lip_\sigma\right)^p
			\E\left(\left| \int_{-\infty}^\infty\d y\int_0^t\d s\,
			\left| f(s\,,y) - g(s\,,y) \right|^2
			\left| p_{t-s}(y-x)\right|^2 \right|^{p/2}\right).
	\end{align}
	We write the expectation as
	\begin{equation}
		\E\left(\prod_{j=1}^{p/2} \int_{-\infty}^\infty\d y_j
		\int_0^t\d s_j\,
		\left| f(s_j\,,y_j) - g(s_j\,,y_j) \right|^2
		\left| p_{t-s_j}(y_j-x)\right|^2  \right),
	\end{equation}
	and apply \eqref{eq:Holder} to obtain the bound
	\begin{align}
		&\E\left(\left| (\sA f)(t\,,x)-(\sA g)(t\,,x)\right|^p\right)\\ \nonumber
		&\le \left(z_p\lip_\sigma\right)^p
			\left(\int_{-\infty}^\infty\d y\int_0^t\d s\,
			\| f(s\,,y) - g(s\,,y) \|_{L^p(\P)}^2
			\left| p_{t-s}(y-x)\right|^2 \right)^{p/2}\\\nonumber
		&\le \left(z_p\lip_\sigma\right)^p
			\|f-g\|_{p,\beta}^p e^{\beta t}
			\left(\int_{-\infty}^\infty\d y\int_0^t\d s\, 
			e^{-2\beta s/p}
			\left| p_s(y-x)\right|^2 \right)^{p/2}.
	\end{align}
	This has the desired effect; see Lemma \ref{lem:monotone}.
\end{proof}

\section{Proofs of the main results}\label{sec:proofs}

\begin{proof}[Proof of Theorem \ref{th:exist}]
	Define
	$v_0(t\,,x):=u_0(x)$ for all $(t\,,x)\in\R_+\times\R$.
	Since $u_0$ is assumed to be bounded, 
	$\|v_0\|_{p,\beta}<\infty$ for all $\beta>0$
	and all even integers $p\ge 2$. Now we
	iteratively set
	\begin{equation}
		v_{n+1}(t\,,x) := (\sA v_n)(t\,,x) + (\sG u_0)(t\,,x)
		\qquad\text{for all $n\ge 0$}.
	\end{equation}

	If we set $\sA v_{-1}:=v_0$, then thanks to Lemma \ref{lem:A1},
	for all $n\ge -1$,
	\begin{equation}\label{eq:Av:v}
		\|\sA v_{n+1}\|_{p,\beta} \le z_p
		\left( |\sigma(0)|+\lip_\sigma \|\sA v_n\|_{p,\beta}\right)
		\sqrt{\Upsilon\left(\frac{2\beta}{p}\right)}.
	\end{equation}
	Since $\lim_{\beta\to\infty}\Upsilon(\beta)=0$,
	we can always choose and fix $\beta>0$ such that
	\begin{equation}\label{good:beta}
		z_p^2\lip_\sigma^2\Upsilon\left(\frac{2\beta}{p}\right)<1.
	\end{equation}
	Given such a $\beta$ we find, after a few lines
	of computation, that
	\begin{equation}
		\sup_{n\ge 0}\|\sA v_n\|_{p,\beta} \le \frac{z_p|\sigma(0)|
		\sqrt{ \Upsilon(2\beta/p)} }{%
		1- z_p\lip_\sigma\sqrt{\Upsilon(2\beta/p)}
		}.
	\end{equation}
	Because $\sG u_0$ is bounded uniformly by $\sup_{z\in\R}u_0(z)$,
	the preceding yields
	\begin{equation}\label{eq:v:sup}
		\sup_{k\ge 1}\| v_k\|_{p,\beta} \le\frac{z_p|\sigma(0)|
		\sqrt{ \Upsilon(2\beta/p)} }{%
		1- z_p\lip_\sigma\sqrt{\Upsilon(2\beta/p)}
		} + \sup_{z\in\R}|u_0(z)|,
	\end{equation}
	which is finite.
	Consequently, Lemma \ref{lem:A2}
	assures us that all $n\ge 1$,
	\begin{equation}\begin{split}
		\|v_{n+1}-v_n\|_{p,\beta} &=
			\|\sA v_n -\sA v_{n-1}\|_{p,\beta}\\
		&\le z_p\lip_\sigma \sqrt{%
			\Upsilon\left(\frac{2\beta}{p}\right)}
			\|v_n-v_{n-1}\|_{p,\beta}.
	\end{split}\end{equation}
	Because of \eqref{good:beta}, this proves 
	the existence of a predictable random field $u$ such that
	$\lim_{n\to\infty} \| v_n-u\|_{p,\beta}=
	\lim_{n\to\infty}\| \sA v_n-\sA u\|_{p,\beta}=0$.
	Consequently, $\|u\|_{p,\beta}<\infty$,
	$\| u - \sA u- \sG u_0\|_{p,\beta} =0$, and 	
	\begin{equation}
		\E\left(\left| u(t\,,x) - (\sA u)(t\,,x)
		-(\sG u_0)(t\,,x) \right|^p\right)=0
		\quad\text{for all $(t\,,x)\in\R_+\times\R$}.
	\end{equation}
	These remarks prove all but one of
	the assertions of the theorem;
	we still need to establish that $u$ is unique up to a modification.
	For that we follow the methods of
	Da Prato \cite{Da}, Da Prato and Zabczyk \cite{DZ}, and
	especially Peszat and Zabczyk \cite{PZ}: Suppose there
	are two solutions $u$ and $\bar{u}$ to \eqref{heat}.
	Define for all predictable random fields $f$, and $T>0$,
	\begin{equation}
		\| f\|_{2,\beta,T}:=\left\{
		\sup_{t\in[0,T]}\sup_{x\in\R} e^{-\beta t}\E\left(
		\left| f(t\,,x) \right|^2\right)
		\right\}^{1/2}.
	\end{equation}
	Then, we can easily modify the proof of 
	Lemma \ref{lem:A2}, using also the fact that $z_2=1$
	[Remark \ref{rem:He}], to deduce that if $u$ and
	$\bar u$ are two solutions to \eqref{heat}, then
	the following holds for all $T>0$:
	\begin{equation}\begin{split}
		\left\| u-\bar{u}\right\|_{2,\beta,T}
			&=\left\| \sA u - \sA \bar u\right\|_{2,\beta,T}\\
		&\le \lip_\sigma
			\sqrt{\Upsilon (\beta)}
			\left\| u-\bar{u}\right\|_{2,\beta,T}.
	\end{split}\end{equation}
	Because $\Upsilon(\beta)$ vanishes as $\beta$ tends to infinity,
	this proves that $\|u-\bar u\|_{2,\beta,T}=0$ for all $T>0$ and all
	sufficiently large $\beta>0$. This implies that
	$u$ and $\bar u$ are modifications of one another.
\end{proof}

\begin{proof}[Proof of Proposition \ref{pr:exist1}]
	Because $c_4:=\sup_{x\in\R}(|\sigma(x)|\vee |u_0(x)|)
	<\infty$,
	the Burkholder--Davis--Gundy implies that
	\begin{equation}\label{eq:BE2}
		\|u(t\,,x)\|_{L^p(\P)} \le c_4+c_4z_p \left(
		\int_0^t \|p_s\|_{L^2(\R)}^2
		\,\d s  \right)^{1/2}.
	\end{equation}
	Therefore, it suffices to prove that
	\begin{equation}\label{eq:goal:BE2}
		\int_0^t \|p_s\|_{L^2(\R)}^2\,\d s=o(t)\quad
		\text{as $t\to\infty$.}
	\end{equation}
	The left-most term is equal
	to $t\int_0^1 \| p_{s t}\|_{L^2(\R)}^2\,\d s$.
	According to \eqref{eq:L2p}, the map
	$s\mapsto\|p_s\|_{L^2(\R)}^2$
	is nonincreasing, and $\lim_{s\to\infty} \|p_s\|_{L^2(\R)}=0$
	by the dominated convergence theorem. Therefore,
	a second appeal to
	the dominated convergence theorem yields \eqref{eq:goal:BE2}
	and hence the theorem.
\end{proof}

\begin{proof}[Proof of Theorem \ref{th:behavior}]
	We aim to prove that
	\begin{equation}\label{goal:BE}
		\int_0^\infty e^{-\beta t} \E\left(\left|
		u(t\,,x)\right|^2\right)\d t=\infty\quad
		\text{provided that $\Upsilon(\beta)\ge q^{-2}$}.
	\end{equation}
	This implies \eqref{eq:BE}, as the following argument shows:
	Suppose, to the contrary, that 
	$\E(|u(t\,,x)|^2)=O(\exp(\alpha t))$ as
	$t\to\infty$, where
	$\Upsilon(\alpha) > q^{-2}$ and $x\in\R$. 
	It follows from this that
	\begin{equation}
		\int_0^\infty e^{-\beta t} \E\left(\left|
		u(t\,,x)\right|^2\right)\d t \le
		\text{const}\cdot\int_0^\infty e^{-(\beta-\alpha)t}\,\d t,
	\end{equation}
	and this is finite
	for every $\beta\in(\alpha\,,\Upsilon^{-1}(q^{-2})).$
	Our finding contradicts \eqref{goal:BE}, and thence follows
	\eqref{eq:BE}. It remains to establish \eqref{goal:BE}.
	
	Let us introduce the following notation:
	\begin{equation}\begin{split}
		F_\beta(x) &:= \int_0^\infty e^{-\beta t}\E
			\left(\left|u(t\,,x)\right|^2\right)\d t\\
		G_\beta(x) &:=\int_0^\infty e^{-\beta t}
			\left|(\breve{p}_t*u_0)(x)\right|^2\d t\\
		H_\beta(x) &:=\int_0^\infty e^{-\beta t}
			\left| p_t(x)\right|^2\d t.
	\end{split}\end{equation}
	
	Because
	\begin{equation}\label{b/c}\begin{split}
		&\E\left(\left| u(t\,,x)\right|^2\right)\\
		&\qquad= \left| \left(\breve{p}_t*u_0\right)(x)\right|^2
			+\int_{-\infty}^\infty\d y
			\int_0^t\d s\,\E\left(\left|\sigma(u(s\,,y))
			\right|^2\right)\left| p_{t-s}(y-x)\right|^2,
	\end{split}\end{equation}
	we may apply Laplace transforms to both sides,
	and then deduce that for all $\beta>0$ and $x\in\R$,
	\begin{equation}\label{eq:BE3}
		F_\beta(x) = G_\beta(x) + \int_{-\infty}^\infty \d y\,
		H_\beta(x-y)
		\int_0^\infty\d s\, e^{-\beta s}\E\left(\left|\sigma(u(s\,,y))
		\right|^2\right).
	\end{equation}
	Because $|\sigma(z)|^2\ge q^2 |z|^2$ for all $z\in\R$,
	we are led to the following:
	\begin{equation}
		F_\beta(x) \ge G_\beta(x) +q^2(F_\beta*H_\beta)(x).
	\end{equation}
	This is a ``renewal inequation,'' and can be solved by
	standard methods. We will spell that argument out carefully,
	since we need an enhanced version shortly:
	If we define the linear operator $\sH$ by
	\begin{equation}
		(\sH f)(x):=q^2\left( H_\beta *f\right)(x),
	\end{equation}
	then we can deduce that
	$\sH^n F_\beta -\sH^{n+1} F_\beta
	\ge \sH^nG_\beta$, pointwise,
	for all integers $n\ge 0$.
	We sum this inequality from $n=0$ to $n=N$
	and find that
	\begin{equation}\begin{split}
		F_\beta(x) &\ge \left( \sH^{N+1}F_\beta \right)(x)+\sum_{n=0}^N
			\left( \sH^nG_\beta \right)(x)\\
		&\ge \sum_{n=0}^N \left( \sH^nG_\beta \right)(x).
	\end{split}\end{equation}
	It follows, upon letting $N$ tend to infinity, that
	\begin{equation}
		F_\beta(x) \ge \sum_{n=0}^\infty 
		( \sH^nG_\beta)(x).
	\end{equation}
	If $\eta:=\inf_x u_0(x)$, then  $(\breve{p}_t*u_0)(x)\ge
	\eta$ pointwise, and hence
	$G_\beta(x)\ge\eta^2/\beta$. Consequently, 
	\begin{equation}\begin{split}
		(\sH G_\beta)(x) &\ge \frac{q^2\eta^2}{\beta}
			\cdot \int_{-\infty}^\infty
			H_\beta(x)\, \d x\\
		&= \frac{q^2\eta^2}{\beta}
			\cdot \Upsilon(\beta);
	\end{split}\end{equation}
	consult Lemma \ref{lem:monotone} for the identity.
	We can iterate the preceding argument to deduce that
	$F_\beta(x) \ge \eta^2\beta^{-1}\sum_{n=0}^\infty
	(q^2 \Upsilon(\beta))^n$,
	whence $F_\beta(x)=\infty$ as long as
	$\Upsilon(\beta)\geq q^{-2}$.  This verifies
	\eqref{goal:BE}, and concludes our proof.
\end{proof}

\begin{proof}[Proof of Proposition \ref{pr:trans}]
	We recall the well-known fact that 
	\begin{equation}\label{Xbar:rec}
		\text{$\bar{X}$ is
		recurrent if and only if}\quad
		\int_{-1}^1 \frac{\d\xi}{\Re\Psi(\xi)}=\infty.
	\end{equation}
	Otherwise, $\bar{X}$ is transient; see Exercise V.6
	of Bertoin \cite[p.\ 152]{Bertoin}.
	Because $\Upsilon(\beta)<\infty$ for all
	$\beta>0$, and since $\Re\Psi(\xi)\ge 0$,
	it is manifest that \eqref{Xbar:rec} is
	equivalent to the following:
	\begin{equation}\label{Xbar:rec1}
		\text{$\bar{X}$ is
		recurrent if and only if}\quad
		\lim_{\beta\downarrow 0}\Upsilon(\beta)=
		\infty.
	\end{equation}
	
	Consequently, when $\bar{X}$ is transient,
	$\sup_{\beta>0}\Upsilon(\beta)=
	\lim_{\beta\downarrow 0}\Upsilon(\beta)<\infty$,
	and the proposition follows immediately from
	Theorem \ref{th:exist}. In fact, we can
	choose $\delta(p)$ to be the reciprocal
	of $z_p \{\sup_{\beta>0}\Upsilon(\beta)\}^{1/2}.$
\end{proof}

\begin{proof}[Proof of Corollary \ref{cor:L:Anderson}]
	Thanks to \eqref{Xbar:rec}, when $\bar{X}$ is recurrent, 
	we can find $\beta>0$ such that $\Upsilon(\beta)>1/\lambda^2$.
	Theorem \ref{th:behavior} implies the exponential
	growth of $u$, and the formula
	for $\bar\gamma(2)$ follows
	upon combining the quantitative bounds of
	Theorems \ref{th:exist} and \ref{th:behavior}.
	The case where $\bar{X}$ is transient is proved similarly.
\end{proof}

We close the paper with the following.

\begin{proof}[Proof of Theorem \ref{th:sublinear}]
	We modify the proof of Theorem \ref{th:behavior}, and point out
	only the requisite changes.
	First of all, let us note that
	for all $q_0\in(0\,,q)$ there exists $A=A(q_0)\in[0\,,\infty)$
	such that $|\sigma(z)|\ge q_0|z|$
	provided that $|z|>A$.
	Consequently, for all $s\in\R_+$ and $y\in\R$,
	\begin{equation}\begin{split}
		\E\left(\left| \sigma(u(s\,,y))\right|^2\right)
			&\ge q_0^2\E\left(\left| u(s\,,y)\right|^2;\,
			|u(s\,,y)| >A\right)\\
		&\ge q_0^2\E\left(\left| u(s\,,y)\right|^2\right)-
			q_0^2A^2.
	\end{split}\end{equation}
	Eq.\ \eqref{b/c} implies that $\E(|u(t\,,x)|^2)$ is bounded below by
	\begin{equation}\begin{split}
		&\left|(\breve{p}_t*u_0)(x)\right|^2+
			q_0^2\int_{-\infty}^\infty\d y
			\int_0^t\d s\,
			\E\left(\left|u(s\,,y)\right|^2\right)
			\left| p_{t-s}(y-x)\right|^2\\
		&\hskip2.5in
			-q_0^2A^2 \int_0^t \left\| p_s\right\|^2_{L^2(\R)}\,\d s.
	\end{split}\end{equation}
	We multiply both sides of
	the preceding display
	by $\exp(-\beta t)$, for a fixed $\beta>0$,
	and integrate $[\d t]$ to find that
	\begin{equation}
		F_\beta(x) \ge G_\beta(x) + (\sH F_\beta)(x)
		- \frac{q_0^2A^2}{\beta}\,\Upsilon(\beta),
	\end{equation}
	where the notation is borrowed from the proof
	of Theorem \ref{th:behavior}. We apply $\sH^n$ to both sides
	to deduce the following: For all integers $n\ge 0$ and
	$x\in\R$,
	\begin{equation}\begin{split}
		(\sH^n F_\beta)(x) &\ge (\sH^n G_\beta)(x) +
			(\sH^{n+1} F_\beta)(x) -
			\frac{q_0^2A^2}{\beta}\,\Upsilon(\beta)\cdot
			\left| q_0^2\Upsilon(\beta)\right|^n\\
		&\ge \frac{\eta^2}{\beta}\cdot
			\left| q_0^2\Upsilon(\beta)\right|^n +
			(\sH^{n+1} F_\beta)(x) -
			\frac{A^2}{\beta}\cdot
			\left| q_0^2\Upsilon(\beta)\right|^{n+1},
	\end{split}\end{equation}
	thanks to the tautological bound $u_0\ge\eta$. We collect terms
	to obtain the following key estimate for the present proof: 
	\begin{equation}
		(\sH^n F_\beta)(x) - (\sH^{n+1}F_\beta)(x) \ge
		\frac{\eta^2  - A^2 q_0^2\Upsilon(\beta)}{\beta}
		\times\left| q_0^2\Upsilon(\beta)\right|^n,
	\end{equation}
	valid for all integers $n\ge 0$ and $x\in\R$.	
	Because $\bar{X}$
	is recurrent, \eqref{Xbar:rec1} ensures that
	we can choose $\beta>0$ sufficiently small that $q_0\Upsilon(\beta)>1$.
	Consequently, $F_\beta(x) \equiv \infty$
	as long as $\eta$ is greater than $Aq_0\Upsilon(\beta)$.
	This proves the theorem;
	confer with the paragraph immediately
	following \eqref{goal:BE}.
\end{proof}
\appendix
\section{Regularity}\label{sec:reg}

The goal of this section is to show that one can produce
a nice modification of the solution to \eqref{heat}.

\begin{theorem}\label{th:reg}
	If $u_0$ is continuous, then the
	solution to \eqref{heat} is continuous in $L^p(\P)$
	for all $p>0$.
	Consequently, $u$ has a separable modification.
	If, in addition, $u_0$ is uniformly continuous,
	then for all $T,p>0$,
	\begin{equation}
		\lim_{\delta,\rho\downarrow 0}\, \sup_{0\le t\le T}\,
		\sup_{x\in\R} \left\| u(t\,,x) - u(s\,,y)
		\right\|_{L^p(\P)}=0.
	\end{equation}
\end{theorem}

This theorem is a ready consequence of the following series of
Lemmas \ref{lem:semigp:reg}, \ref{lem:Areg1},
\ref{lem:Areg2}, and \ref{lem:Areg3}. We sketch [most of]
the proofs because
many of the methods of this section merely expand on those
of the earlier sections.

Let $\varpi$ denote the uniform modulus of continuity of $u_0$.
That is,
\begin{equation}
	\varpi(\delta) := \sup_{|a-b|<\delta}
	\left| u_0(a)-u_0(b) \right|.
\end{equation}

\begin{lemma}\label{lem:semigp:reg}
	If $u_0$ is continuous, then so is
	$(t\,,x)\mapsto(\mathcal{P}_tu_0)(x)$.
	If $u_0$ is uniformly continuous,
	then so is $(t\,,x)\mapsto (\sP_t u_0)(x)$; in fact
	for all $\delta,\rho>0$,
	\begin{equation}\label{eq:semigp:reg:1}
		\sup_{t\ge 0}
		\sup_{|x-z|\le\delta}\left| (\sP_tu_0)(x)-(\sP_tu_0)(z)\right|
		\le \varpi(\delta),
	\end{equation}
	and
	\begin{equation}\label{eq:semigp:reg:2}
		\sup_{|t-s|<\rho}\,
		\sup_{x\in\R}\left| (\sP_t \,u_0)(x)-(\sP_s u_0)(x)\right|
		\le \inf_{a>0}\left[ \varpi(a) + A\rho
		\sup_{0<\xi <1/a} |\Psi(\xi)|\right],
	\end{equation}
	with $A:=14\sup_{z\in\R} u_0(z)$.
\end{lemma}

\begin{proof}
	We note that
	\begin{equation}
		(\mathcal{P}_tu_0)(x)-(\mathcal{P}_su_0)(y)
		=\E\left( u_0(X_t+x) - u_0(X_s+y) \right).
	\end{equation}
	Because $u_0$ is bounded, if it were continuous
	also, then $(t\,,x)\mapsto(\mathcal{P}_tu_0)(x)$
	is continuous by the dominated convergence theorem.
	Henceforth, we assume that $u_0$ is uniformly continuous.
	Inequality \eqref{eq:semigp:reg:1} follows again from
	the dominated convergence theorem.
	As regards \eqref{eq:semigp:reg:2}, we note that
	\begin{equation}
		\sup_{x\in\R}\left| (\sP_t \,u_0)(x)-(\sP_s u_0)(x)\right|
		\le \E\left[ \varpi\left( \left| X_t-X_s \right| \right) \wedge 
		2\sup_{z\in\R} u_0(z) \right].
	\end{equation}
	Because $| 1-\E \exp(i\xi (X_t-X_s))|
	\le |t-s|\cdot|\Psi(\xi)|$, Paul
	L\'evy's characteristic-function inequality
	\cite[Exercise 7.9, p.\ 112]{Kh} shows that for all
	$a>0$,
	\begin{equation}\begin{split}
		\P\left\{ |X_t-X_s|>a\right\}
			&\le 7a\int_0^{1/a}\left| 1-\E e^{i\xi (X_t-X_s)} \right|\,\d\xi\\
		&\le 7|t-s|\sup_{0<\xi <1/a}|\Psi(\xi)|.
	\end{split}\end{equation}
	This completes our proof readily.
\end{proof}

\begin{lemma}\label{lem:Areg1}
	For all even integers $p\ge 2$,
	$x,z\in\R$, $t\ge 0$, and $\beta>0$,
	\begin{equation}\label{eq:Areg1}\begin{split}
		&\left\| (\sA u)(t\,,x) - (\sA u)(t\,,z) \right\|_{L^p(\P)}\\
		&\hskip.7in\le \left(\frac{p}{\pi}\right)^{1/2}
			\|\sigma\circ u\|_{p,\beta}\, e^{t\beta/p}
			\left\{\int_{-\infty}^\infty
			\frac{1-\cos(\xi|x-z|)}{\beta+2\Re\Psi(\xi)}
			\,\d\xi \right\}^{1/2},
	\end{split}\end{equation}
	where $(\sigma\circ u)(t\,,x):=\sigma(u(t\,,x))$.
\end{lemma}

\begin{proof}
	We follow the pattern of the proof of Lemma \ref{lem:A2},
	and, after a few lines of estimates, deduce that
	\begin{align}
		&\left\| (\sA u)(t\,,x) - (\sA u)(t\,,z) \right\|_{L^p(\P)}^p\\\nonumber
		&\qquad \le z_p^p \left| \int_{-\infty}^\infty\d y
			\int_0^t \d s\,
			\|\sigma(u(s\,,y))\|_{L^p(\P)}^2 \left[
			p_{t-s}(y-x)-p_{t-s}(y-z) \right]^2 \right|^{p/2}.
	\end{align}
	Since $\|\sigma(u(s\,,y))\|_{L^p(\P)}^2\le\exp(2s\beta/p)\|\sigma
	\circ u\|_{p,\beta}^2$ for all $\beta>0$, the preceding and
	the Carlen--Kree inequality (Remark \ref{rem:He}) together yield
	\begin{align}
		&\left\| (\sA u)(t\,,x) - (\sA u)(t\,,z)
			\right\|_{L^p(\P)}^p\\\nonumber
		&\qquad \le 2^p
			p^{p/2}\, \|\sigma\circ u\|_{p,\beta}^2\, e^{\beta t}
			\left| \int_{-\infty}^\infty\d y
			\int_0^t \d s\, e^{-2s\beta/p}\left[
			p_s(y-x)-p_s(y-z) \right]^2 \right|^{p/2}.
	\end{align}
	In accord with Plancherel's theorem,
	\begin{equation}\begin{split}
		&\int_{-\infty}^\infty \left[ p_s(y-x)-p_s(y-z) \right]^2\,\d y\\
		&\hskip1.5in =\frac{1}{\pi}\int_{-\infty}^\infty
			\left( 1-\cos\left(\xi|x-z|\right)
			\right) e^{-2s\Re\Psi(\xi)}\,\d \xi.
	\end{split}\end{equation}
	The lemma follows from this and a few more lines of computation.
\end{proof}

Choose $x\in\R$ and $0\le t\le T$. We can write
\begin{equation}
	(\sA u)(T,x) - (\sA u)(t\,,x) = D_1+D_2,
\end{equation}
where
\begin{equation}\begin{split}
	D_1 &:= \int_{-\infty}^\infty\int_0^t \sigma(u(s\,,y))
		\left[ p_{T-s}(y-x)-p_{t-s}(y-x) \right]\,
		w(\d s\,\d y),\\
	D_2 &:= \int_{-\infty}^\infty\int_t^T \sigma(u(s\,,y))
		p_{T-s}(y-x)\,w(\d s\,\d y).
\end{split}\end{equation}

\begin{lemma}\label{lem:Areg2}
	For all even integers $p\ge 2$ and $\beta>0$,
	\begin{equation}\label{eq:Areg2}
		\|D_1\|_{L^p(\P)}\le e^{\beta t/p}
		\left(\frac{p}{\pi}\right)^{1/2}\, \|\sigma\circ u\|_{p,\beta}
		\left(\int_{-\infty}^\infty
		\frac{\left| 1 - e^{-(T-t)\Psi(\xi)}
		\right|^2}{(\beta/p)+\Re\Psi(\xi)}\,\d\xi\right)^{1/2}.
	\end{equation}
\end{lemma}

\begin{proof}
	We adjust the beginning portion of the preceding proof,
	and after a few lines, arrive at the following:
	\begin{equation}\label{eq:5.13}
		\E(D_1^p) \le 2^pp^{p/2}\|\sigma\circ u\|_{p,\beta}^p
		\left(\int_0^t e^{2\beta s/p}\left\|
		p_{T-s}-p_{t-s} \right\|_{L^2(\R)}^2\,\d s\right)^{p/2}.
	\end{equation}
	Thanks to Plancherel's theorem,
	\begin{equation}
		\left\| p_{T-s}-p_{t-s} \right\|_{L^2(\R)}^2
		=\frac{1}{2\pi}\int_{-\infty}^\infty
		e^{-2(t-s)\Re\Psi(\xi)}\left| 1 - e^{-(T-t)\Psi(\xi)}
		\right|^2\,\d\xi.
	\end{equation}
	Therefore, the Lebesgue integral in \eqref{eq:5.13} is bounded
	above by
	\begin{equation}
		\frac{e^{\beta t}}{2\pi} \int_{-\infty}^\infty
		\frac{\left| 1 - e^{-(T-t)\Psi(\xi)}
		\right|^2}{(2\beta/p)+2\Re\Psi(\xi)}\,\d\xi.
	\end{equation}
	Solve to finish.
\end{proof}

\begin{lemma}\label{lem:Areg3}
	For all even integers $p\ge 2$ and $\beta>0$,
	\begin{equation}\label{eq:Areg3}
		\|D_2\|_{L^p(\P)}\le\sqrt{8p}\, e^{\beta T/p}\,
		\|\sigma\circ u\|_{p,\beta}\sqrt{\Upsilon\left(
		\frac{1}{T-t}\right)}.
	\end{equation}
\end{lemma}

\begin{proof}
	We adapt the proof of the preceding lemma,
	to the present setting, and deduce that
	\begin{equation}
		\|D_2\|_{L^p(\P)}\le
		2\sqrt{p}\,\|\sigma\circ u\|_{p,\beta}\, e^{\beta T/p}
		\left( \int_0^{T-t}\|p_s\|_{L^2(\R)}^2\,\d s\right)^{1/2}.
	\end{equation}
	But for all $\rho>0$,
	\begin{equation}\begin{split}
		\int_0^\rho \|p_s\|_{L^2(\R)}^2\,\d s
			&= \frac{1}{2\pi}\int_{-\infty}^\infty \frac{1-e^{-2
			\rho\Re\Psi(\xi)}}{2\Re\Psi(\xi)}\,\d\xi\\
		&\le \frac{1}{\pi}\int_{-\infty}^\infty
			\frac{\d\xi}{(1/\rho)+2\Re\Psi(\xi)}\\
		&=2\Upsilon(1/\rho).
	\end{split}\end{equation}
	We have used the elementary fact that
	$(1-e^{-\rho\theta})/\theta\le 2/(\rho^{-1}+\theta)$
	for all $\theta>0$. The lemma follows easily
	from these observations.
\end{proof}

One can often combine the preceding proof of Theorem \ref{th:reg}
with methods of Gaussian analysis and produce an {\it almost surely
continuous} modification of $u$. We conclude this paper with
an example of this method.

\begin{example}\label{ex:Stable:Holder}
	Suppose $1<\alpha\le 2$ and $\sL=-\kappa(-\Delta)^{\alpha/2}$.
	Suppose also that $u_0$ is uniformly H\"older continuous; that is,
	$\varpi(a)=O(a^\theta)$ as $a\to 0^+$ for a fixed $\theta>0$.
	We claim that in this case $u$ has a modification that is
	continuous almost surely. We prove this claim by working out
	the estimates produced by Lemmas \ref{lem:semigp:reg},
	\ref{lem:Areg1}, \ref{lem:Areg2}, and \ref{lem:Areg3}.
	Indeed, \eqref{eq:semigp:reg:1} and \eqref{eq:semigp:reg:2}
	together show that $(t\,,x)\mapsto(\sP_t u_0)(x)$ is
	uniformly jointly H\"older continuous with respective
	H\"older indices $\upsilon:=\theta/(\theta+\alpha)$
	[for $t$] and $\theta$ [for $x$]. A few 
	more simple calculations
	show that: (i) The right-hand side of \eqref{eq:Areg1} is
	$O(|x-z|^\mu)$ with $\mu:=\min(1/2\,,\alpha-1)$; and
	(ii) the right-hand sides of \eqref{eq:Areg2} and
	\eqref{eq:Areg3} are both 
	$O((T-t)^\eta)$ with $\eta:=(\alpha-1)/(2\alpha)$.
	In other words, we can choose and fix $\beta>0$
	that yields the following estimate: For all $T,p>0$
	there exists $a=a(p\,,T,\beta)\in(0\,,\infty)$ such that
	for all $s,t\in[0\,,T]$ and $x,y\in\R$,
	\begin{equation}
		\| u(t\,,x)-u(s\,,y) \|_{L^p(\P)} \le a
		\left( |t-s|^{\upsilon\wedge\eta} +|x-y|^{\theta\wedge\mu}
		\right).
	\end{equation}
	A suitable form of the Kolmogorov continuity theorem
	yields the desired H\"older-continuous modification.
	\qed
\end{example}

\begin{small}

\vskip.4cm

\noindent\textbf{Mohammud Foondun} \& \textbf{Davar Khoshnevisan}\\
\noindent Department of Mathematics, University of Utah,
		Salt Lake City, UT 84112-0090\\
\noindent\emph{Emails:} \texttt{mohammud@math.utah.edu} \&
	\texttt{davar@math.utah.edu}\\
\noindent\emph{URLs:} \texttt{http://www.math.utah.edu/\~{}mohammud} \&
	\texttt{http://www.math.utah.edu/\~{}davar}
\end{small}

\end{document}